\documentclass{amsart}
\usepackage{amssymb,amsmath,latexsym}
\usepackage{amsthm}
\usepackage{fontenc}
\usepackage{amssymb}
\numberwithin{equation}{section}
 \DeclareMathOperator{\erf}{erf}

\newtheorem{theorem}{Theorem}[section]

\newtheorem{lemma}[theorem]{Lemma}

\begin{document}
\author{Alexander E. Patkowski}
\title{On asymptotic expansions for basic hypergeometric functions }

\maketitle
\begin{abstract} This paper establishes new results concerning asymptotic expansions of $q$-series related to partial theta functions. We first establish a new method to obtain asymptotic expansions using a result of Ono and Lovejoy, and then build on these observations to obtain asymptotic expansions for related multi hypergeometric series. \end{abstract}

\keywords{\it Keywords: \rm $L$-function; $q$-series; asymptotics}

\subjclass{ \it 2010 Mathematics Subject Classification Primary 33D90; Secondary 11M41}
\section*{Acknowledgement: Dedicated to Julia, Carolyn, and Piotr.}

\section{Introduction and Main Results}

\par In keeping with usual notation [6], put $(y)_n=(y;q)_{n}:=\prod_{0\le k\le n-1}(1-yq^{k}),$ and define $(y)_{\infty}=(y;q)_{\infty}:=\lim_{n\rightarrow\infty}(y;q)_{n}.$ Recall (see [11, pg.182]) the parabolic cylinder function $D_s(x),$ defined as the solution to the differential equation of Weber [7, pg.1031, eq.(9.255), \#1]
$$\frac{\partial^2 f}{\partial x^2}+(s+\frac{1}{2}-\frac{x^2}{4})f=0.$$ This function also has a relationship with the confluent hypergeometric function ${}_1F_1(a;b; x)$ [11, pg.183], [7, pg.1028]
$$D_s(x)=e^{-x^2/4}\sqrt{\pi}\left(\frac{2^{s/2}}{\Gamma(\frac{1}{2}-\frac{s}{2})}{}_1F_1(-\frac{s}{2};\frac{1}{2};\frac{x^2}{2})-\frac{x2^{s/2+1/2}}{\Gamma(-\frac{s}{2})}{}_1F_1(\frac{1}{2}-\frac{s}{2};\frac{3}{2};\frac{x^2}{2})\right).$$
In the last two decades considerable attention has been given to asymptotic expansions of basic hypergeometric series when $q=e^{-t}$ as $t\rightarrow0^{+}.$ Special examples related to negative values of $L$-functions have been offered in [3, 8, 9, 12, 13, 14]. Recall that an $L$-function is defined as the series $L(s)=\sum_{n\ge1}a_nn^{-s},$ for a suitable arithmetic function $a_n:\mathbb{N}\rightarrow\mathbb{C}.$ In the simple case $a_n=1,$ we have the Riemann zeta function $L(s)=\zeta(s),$ for $\Re(s)>1,$ and if $a_n=(-1)^n,$ we obtain $L(s)=(1-2^{1-s})\zeta(s).$
\par As it turns out, there are general expansions for basic hypergeometric series with $q=e^{-t}$ which include special functions as terms along with values of $L(s).$ To this end, we offer the first instance of a more general expansion in the literature for certain multi-dimensional basic hypergeometric series. To illustrate our approach we first recall the result due to Ono and Lovejoy [9, Theorem 1], which says that for $k\ge2,$ integers $0<m<l,$ if
\begin{equation}F_k(z,q):=\sum_{r_{k-1}\ge r_{k-2}\dots\ge r_1\ge0}\frac{(q)_{n_{k-1}}(z)_{n_{k-1}}z^{n_{k-1}+2n_{k-2}+\dots+2n_1}q^{r_1^2+r_1+r_2^2+r_2+\dots+r_k^2+r_k}}{(q)_{r_{k-1}-r_{k-2}}(q)_{r_{k-2}-r_{k-3}}\dots(q)_{r_{2}-r_{1}}(q)_{r_1}(-z)_{r_1+1}},\end{equation}
then as $t\rightarrow0^{+},$
$$e^{-(k-1)m^2t}F_k(e^{-lmt},e^{-l^2t})=\sum_{n\ge0}L_{l,m}(-2n)\frac{((1-k)t)^n}{n!}, $$
where $L_{l, m}(s)=(2l)^{-s}\left(\zeta(s,\frac{m}{2l})-\zeta(s,\frac{l+m}{2l})\right).$ Here the Hurwitz zeta function is $\zeta(s,x)=\sum_{n\ge0}(n+x)^{-s},$ a one parameter refinement of $\zeta(s).$ Their paper uses a clever specialization of Andrews' refinement of a transformation of Watson to obtain (1.1). By [9, Theorem 2.1] 
\begin{equation} F_k(z,q)=\sum_{n\ge0}(-1)^nz^{(2k-2)n}q^{(k-1)n^2}.\end{equation}

Putting $q=e^{-wt^2},$ and $z=e^{-vt},$ and taking the Mellin transform (see (2.10) or [11]) of (1.2), Lemma 2.4 of the next section tells us that for $\Re(s)>0,$ $\Re(w)>0,$

\begin{equation}\int_{0}^{\infty}t^{s-1} \left(F_k(e^{-vt},e^{-wt^2})-1\right)dt\end{equation}
$$= (2w(k-1))^{-s/2}(1-2^{1-s})\zeta(s)\Gamma(s)e^{v^2(k-1)/(2w)}D_{-s}\left(\frac{v2(k-1)}{\sqrt{2(k-1)w}}\right).$$
Now by Mellin inversion and Cauchy's residue theorem, it can be shown that
\begin{equation}F_k(e^{-vt},e^{-wt^2})-1\end{equation}
$$\sim e^{v^2(k-1)/(2w)}\sum_{n\ge0}\frac{(2w(k-1))^{n/2}}{n!}(-t)^n(1-2^{1+n})\zeta(-n)D_{n}\left(\frac{v2(k-1)}{\sqrt{2(k-1)w}}\right),$$
as $t\rightarrow0^{+}.$ Here we employed the standard notation $f(x)\sim g(x)$ as $x\rightarrow x_0$ to imply that $\lim_{x\rightarrow x_0}f(x)/g(x) = 1.$ As a result of our new observation, we are able to produce general expansions of a similar type as (1.4) by appealing to Bailey chains [1]. We believe our observation is significant in its implications for further asymptotic expansions for basic hypergeometric series. Additionally, our expansions generalize the classical case for the Jacobi theta function. Recall [11, pg.353, eq.(8.1.2)], $\Re(t)>0,$
$$\sum_{n\ge0}e^{-(tn)^2}=\frac{\sqrt{\pi}}{2t}+\frac{1}{2}+\frac{\sqrt{\pi}}{t}\sum_{n\ge1}e^{-(\pi n/t)^2}, $$
which implies that as $t\rightarrow 0^{+},$
$$-\frac{1}{2}+\sum_{n\ge0}e^{-(tn)^2}\sim\frac{\sqrt{\pi}}{2t}.$$ This observation is a key property in establishing the Riemann integral representation for $\pi^{-s/2}\Gamma(\frac{s}{2})\zeta(s),$ valid for the entire complex plane $\mathbb{C}.$ To see this, note from [15, pg.22] that this asymptotic formula establishes convergence of the integral over $[0,1]$ in
$$\pi^{-s/2}\Gamma(\frac{s}{2})\zeta(s)=2\int_{0}^{1}t^{s-1}\left(\sum_{n\ge1}e^{-\pi(t n)^2}\right)dt+2\int_{1}^{\infty}t^{s-1}\left(\sum_{n\ge1}e^{-\pi(t n)^2}\right)dt.$$

It is also noted in [11, pg.353] that when $\sqrt{t}$ is purely imaginary, and the sum is truncated at $n=N,$ there is an application to optical diffraction. The classical theta function asymptotics may be recovered from our theorems when letting the parameter $v$ tend to $0.$ The expansions provided in the following should be compared with the general expansion from the Euler-Maclaurin summation formula [11, pg.119, eq.(4.1.5)].

\begin{theorem}\label{thm:theorem1} Define for $k\ge1$
$$A_{n,k}(z,q):=\sum_{n\ge r_1\ge r_2\dots\ge r_k\ge0}\frac{a^{r_1+r_2+\dots+r_k}q^{r_1^2+r_2^2+\dots+r_k^2}}{(q)_{n-r_1}(q)_{r_1-r_2}\dots(q)_{r_{k-1}-r_{k}}}\frac{(z)_{r_k+1}(q/z)_{r_k}}{(q)_{2r_k+1}}.$$
For any $v\in\mathbb{C},$ and $\Re(w)>0,$
$$-1+\sum_{n\ge0}(e^{-wt^2};e^{-wt^2})_n(-1)^ne^{-wt^2n(n+1)/2}A_{n,k}(e^{-vt-wt^2(k+1)},e^{-wt^2})$$
$$\sim t^{-1}\sqrt{\frac{\pi}{4w(k+1)}}e^{v^2/(4(k+1)w)}\bigg(\erf\left(\frac{v}{2\sqrt{(k+1)w}}\right)-\erf\left(-\frac{v}{2\sqrt{(k+1)w}}\right)\bigg)$$
$$+2e^{v^2/(8(k+1)w)}\sum_{n\ge0}\frac{(2w(k+1))^{(2n+1)/2}}{(2n+1)!}t^{2n+1}\zeta(-2n-1)D_{2n+1}\left(\frac{v}{\sqrt{2(k+1)w}}\right),$$
as $t\rightarrow 0^{+}.$
\end{theorem}

We mention that the function in Theorem 1.1. should be compared to the $G(z,q)$ function contained in [9].

\begin{theorem}\label{thm:theorem2} Define for $k\ge1$
$$B_{n,k}(z,q):=\frac{1}{1-q}\sum_{n\ge r_1\ge r_2\dots\ge r_k\ge0}P_{n,r_1,r_2,\dots,r_k}(q)q^{n-r_1+2(r_1-r_2)+\dots+2^{k-1}(r_{k-1}-r_k)},$$
where
$$P_{n,r_1,r_2,\dots,r_k}(q):=\frac{(-q^2;q)_{2r_1}(-q^{2^2};q^{2})_{2r_2}\dots(-q^{2^{k}};q^{2^{k-1}})_{2r_k}}{(q^2;q^2)_{n-r_1}(q^{2^2};q^{2^2})_{r_1-r_2}\dots(q^{2^{k}};q^{2^{k}})_{r_{k-1}-r_{k}}}\frac{(z;q^{2^k})_{r_k+1}(q^{2^k}/z;q^{2^k})_{r_k}}{(q^{2^k+1};q^{2^k})_{2r_k}}.$$
For any $v\in\mathbb{C},$ and $\Re(w)>0,$
$$-1+\sum_{n\ge0}(e^{-wt^2};e^{-wt^2})_n(-1)^ne^{-wt^2n(n+1)/2}B_{n,k}(e^{-vt-wt^2(2^{k}+1)},e^{-wt^2})$$
$$\sim t^{-1}\sqrt{\frac{\pi}{4w(2^k+1)}}e^{v^2/(4(2^k+1)w)}\bigg(\erf\left(\frac{v}{2\sqrt{(2^k+1)w}}\right)-\erf\left(-\frac{v}{2\sqrt{(2^k+1)w}}\right)\bigg)$$
$$+2e^{v^2/(4(k+1)w)}\sum_{n\ge0}\frac{(w(2^{k}+1))^{(2n+1)/2}}{(2n+1)!}t^{2n+1}\zeta(-2n-1)D_{2n+1}(\frac{v}{\sqrt{(2^{k}+1)w}}),$$
as $t\rightarrow 0^{+}.$
\end{theorem}

\begin{theorem}\label{thm:theorem3} Define for $k\ge1$
$$C_{n.k}(z,q)=\sum_{n\ge r_1\ge r_2\dots\ge r_k\ge0}\frac{q^{r_1^2/2+r_1/2+r_2^2/2+r_2/2+\dots+r_k^2/2+r_k/2}}{(q)_{n-r_1}(q)_{r_1-r_2}\dots(q)_{r_{k-1}-r_{k}}}\frac{(z)_{r_k+1}(q/z)_{r_k}}{(q)_{{r_k}}(q;q^2)_{r_k+1}}.$$
For any $v\in\mathbb{C},$ and $\Re(w)>0,$
$$-1+\sum_{n\ge0}\frac{(e^{-wt^2};e^{-wt^2})_n}{(-e^{-wt^2};e^{-wt^2})_n}(-1)^ne^{-wt^2n(n+1)/2}C_{n,k}(e^{-vt-wt^2(k+1)/2},e^{-wt^2})$$
$$\sim t^{-1}\sqrt{\frac{\pi}{2w(k+2)}}e^{v^2/(2(k+2)w)}\bigg(\erf\left(\frac{v}{2\sqrt{(k+2)w}}\right)-\erf\left(-\frac{v}{2\sqrt{(k+2)w}}\right)\bigg)$$
$$+2e^{v^2/(4(k+2)w)}\sum_{n\ge0}\frac{(w(k+2))^{(2n+1)/2}}{(2n+1)!}t^{2n+1}\zeta(-2n-1)D_{2n+1}(\frac{v}{\sqrt{(k+2)w}}),$$
as $t\rightarrow0^{+}.$
\end{theorem}
We note that it may be desirable to utilize the Hermite polynomial representation of the parabolic cylinder function [7, pg.1030, eq.(9.253)] $D_n(x)=2^{-n/2}e^{-x^2/4}H_n(\frac{x}{\sqrt{2}}),$ as an alternative. The single sum in our expansions arise from the parity relationship $D_n(-x)=(-1)^nD_n(x),$ resulting in some collapsing for Theorems 1.1 through Theorem 1.3. \par Our last result is an apparently new asymptotic expansion for a partial theta function involving a nonprincipal Dirichlet character $\chi.$

\begin{theorem}\label{thm:theorem4} Let $\chi(n)$ be a real, primitive, nonprincipal Dirichlet character associated with $L(s,\chi).$ Let $v\in\mathbb{C},$ and $\Re(w)>0.$ Then if $\chi$ is an even character,
$$\sum_{n\ge1}\chi(n)e^{-w n^2t^2-v nt}\sim e^{v^2/(8w)}\sum_{n\ge0}\frac{(2w)^{(2n+1)/2}(-t)^{2n+1}}{(2n+1)!}L(-2n-1,\chi)D_{2n+1}(\frac{v}{\sqrt{2w}}),$$
as $t\rightarrow0^{+},$
and if $\chi$ is an odd character,
$$\sum_{n\ge1}\chi(n)e^{-w n^2t^2-v nt}\sim e^{v^2/(8w)}\sum_{n\ge0}\frac{(2w)^{n}(-t)^{2n}}{(2n)!}L(-2n,\chi)D_{2n}(\frac{v}{\sqrt{2w}}),$$
as $t\rightarrow0^{+}.$

\end{theorem}

\section{Proofs of results}

We recall that a pair $(\alpha_n(a, q),\beta_n(a,q))$ is referred to as a Bailey pair [4] with respect to $(a,q)$ if
\begin{equation}\beta_n(a,q)=\sum_{0\le j\le n}\frac{\alpha_j(a,q)}{(q;q)_{n-j}(aq;q)_{n+j}}.\end{equation}
We will use Bailey's lemma [5, pg.2] to obtain our basic hypergeometric identities, which states that if $(\alpha_n(a, q),\beta_n(a,q))$ is a Bailey pair with respect to $(a,q),$ then $(\alpha'_n(a, q),\beta'_n(a,q))$ is a Bailey pair with respect to $(a,q)$ where 

$$\alpha'_n(a, q)=\frac{(\rho_1)_n(\rho_2)_n}{(aq/\rho_1)_{n}(aq/\rho_2)_{n}}(aq/\rho_1\rho_2)^n\alpha_n(a, q),$$
$$\beta'_n(a,q)=\sum_{k\ge0}\frac{(\rho_1)_k(\rho_2)_k(aq/\rho_1\rho_2)_{n-k}(aq/\rho_1\rho_2)^k\beta_k(a,q)}{(aq/\rho_1)_n(aq/\rho_2)_n(q)_{n-k}}.$$

Iterating [5, (S1)] $k$ times gives us the following lemma.
\begin{lemma} For $k\ge1,$ if $(\alpha_n(a, q),\beta_n(a,q))$ is a Bailey pair with respect to $(a,q),$ then $(\alpha'_n(a, q),\beta'_n(a,q))$ is a Bailey pair with respect to $(a,q)$ where 
\begin{equation} \beta'_n(a,q)=\sum_{n\ge r_1\ge r_2\dots\ge r_k\ge0}\frac{a^{r_1+r_2+\dots+r_k}q^{r_1^2+r_2^2+\dots+r_k^2}}{(q)_{n-r_1}(q)_{r_1-r_2}\dots(q)_{r_{k-1}-r_{k}}}\beta_{r_k}(a,q)\end{equation}
\begin{equation} \alpha'_n(a,q)=a^{kn}q^{kn^2}\alpha_n(a,q).\end{equation}
\end{lemma}
Iterating [5, (D1)] $k$ times gives us the following different lemma.
\begin{lemma} For $k\ge1,$ if $(\alpha_n(a, q),\beta_n(a,q))$ is a Bailey pair with respect to $(a,q),$ then $(\alpha'_n(a, q),\beta'_n(a,q))$ is a Bailey pair with respect to $(a,q)$ where 
\begin{equation} \beta'_n(a,q)=\sum_{n\ge r_1\ge r_2\dots\ge r_k\ge0}p_{n,r_1,r_2,\dots,r_k}(a,q)q^{n-r_1+2(r_1-r_2)+\dots+2^{k-1}(r_{k-1}-r_k)}\beta_{r_k}(a^{2^{k}},q^{2^{k}}),\end{equation}
where
$$p_{n,r_1,r_2,\dots,r_k}(a,q):=\frac{(-aq;q)_{2r_1}(-a^{2}q^{2};q^{2})_{2r_2}\dots(-a^{2^{k}}q^{2^{k-1}};q^{2^{k-1}})_{2r_k}}{(q^2;q^2)_{n-r_1}(q^{2^2};q^{2^2})_{r_1-r_2}\dots(q^{2^{k}};q^{2^{k}})_{r_{k-1}-r_{k}}},$$
and
\begin{equation} \alpha'_n(a,q)=\alpha_n(a^{2^{k}},q^{2^{k}}).\end{equation}
\end{lemma}
Lastly, we iterate [5, (S2)] $k$ times.
\begin{lemma} For $k\ge1,$ if $(\alpha_n(a, q),\beta_n(a,q))$ is a Bailey pair with respect to $(a,q),$ then $(\alpha'_n(a, q),\beta'_n(a,q))$ is a Bailey pair with respect to $(a,q)$ where 
\begin{equation} \beta'_n(a,q)=\frac{1}{(-\sqrt{aq})_{n}}\sum_{n\ge r_1\ge r_2\dots\ge r_k\ge0}\frac{(-\sqrt{aq})_{r_k}a^{r_1/2+r_2/2+\dots+r_k/2}q^{r_1^2/2+r_2^2/2+\dots+r_k^2/2}}{(q)_{n-r_1}(q)_{r_1-r_2}\dots(q)_{r_{k-1}-r_{k}}}\beta_{r_k}(a,q)\end{equation}
\begin{equation} \alpha'_n(a,q)=a^{kn/2}q^{kn^2/2}\alpha_n(a,q).\end{equation}

\end{lemma}

Now it is well-known [1. Lemma 6] that $(\alpha_n(q,q),\beta_n(q,q))$ form a Bailey pair where 

\begin{equation}\alpha_n(q,q)=(-z)^{-n}\frac{q^{\binom{n+1}{2}}(1-z^{2n+1})}{1-q},\end{equation}
\begin{equation}\beta_n(q,q)=\frac{(z)_{n+1}(q/z)_{n}}{(q)_{2n+1}}.\end{equation}
The Mellin transform is defined as [11] (assuming $g$ satisfies certain growth conditions)
\begin{equation}\mathfrak{M}(g)(s):=\int_{0}^{\infty}t^{s-1}g(t)dt.\end{equation} The main integral formula we will utilize is given in [7, pg.365, eq.(3.462), \#1].
\begin{lemma} For $\Re(s)>0,$ $\Re(w)>0,$
$$\int_{0}^{\infty}t^{s-1}e^{-w t^2-v t}dt=(2w)^{-s/2}\Gamma(s)e^{v^2/(8w)}D_{-s}(\frac{v}{\sqrt{2w}}).$$

\end{lemma}
By the Legendre duplication formula $\Gamma(\frac{s}{2})\Gamma(\frac{s}{2}+\frac{1}{2})=2^{1-s}\sqrt{\pi}\Gamma(s),$ and the value $D_{-s}(0)\Gamma(\frac{1+s}{2})=2^{s/2}\sqrt{\pi},$ it is readily observed that the $v\rightarrow0$ case of this lemma reduces to the integral formula for $w^{-s/2}\Gamma(\frac{s}{2}).$ The parabolic cylinder function is an analytic function in $v$ and $x$ that enjoys the property that it has no singularities.

\begin{proof}[Proof of Theorem~\ref{thm:theorem1}] 
Inserting (2.8)--(2.9) into Lemma 2.1 gives us the following Bailey pair

\begin{equation}\bar{\beta}_n(q,q):=\sum_{n\ge r_1\ge r_2\dots\ge r_k\ge0}\frac{a^{r_1+r_2+\dots+r_k}q^{r_1^2+r_2^2+\dots+r_k^2}}{(q)_{n-r_1}(q)_{r_1-r_2}\dots(q)_{r_{k-1}-r_{k}}}\frac{(z)_{r_k+1}(q/z)_{r_k}}{(q)_{2r_k+1}}\end{equation}
\begin{equation}\bar{\alpha}_n(q,q)=(-z)^{-n}\frac{q^{(2k+1)n(n+1)/2}(1-z^{2n+1})}{1-q},\end{equation}
A limiting case of Bailey's lemma [1, pg.270, eq.(2.4)] (with $a=q,$ $\rho_1=q,$ $\rho_2\rightarrow\infty,$ and $N\rightarrow\infty$) says that
\begin{equation} \sum_{n\ge0}(q)_n(-1)^nq^{n(n+1)/2}\beta_n(q,q)=(1-q)\sum_{n\ge0}(-1)^n q^{n(n+1)/2}\alpha_n(q,q).\end{equation}
Inserting (2.11)--(2.12) into (2.13) gives
\begin{equation} \begin{aligned}&\sum_{n\ge0}(q)_n(-1)^nq^{n(n+1)/2}\sum_{n\ge r_1\ge r_2\dots\ge r_k\ge0}\frac{a^{r_1+r_2+\dots+r_k}q^{r_1^2+r_2^2+\dots+r_k^2}}{(q)_{n-r_1}(q)_{r_1-r_2}\dots(q)_{r_{k-1}-r_{k}}}\frac{(z)_{r_k+1}(q/z)_{r_k}}{(q)_{2r_k+1}}\\
&=\sum_{n\ge0}(-1)^n(-z)^{-n}q^{(k+1)n(n+1)}(1-z^{2n+1}).\end{aligned} \end{equation}

Putting $q=e^{-wt^2},$ and $z=e^{-vt-wt^2(k+1)},$ we have that (2.14) becomes
\begin{equation} \begin{aligned} &\sum_{n\ge0}(e^{-wt^2};e^{-wt^2})_ne^{-wt^2n(n+1)/2}(-1)^nA_{n,k}(e^{-vt-wt^2(k+1)},e^{-wt^2})\\
&=\sum_{n\ge0}e^{nvt-wt^2(k+1)n^2}(1-e^{-vt(2n+1)-wt^2(k+1)(2n+1)})\\
&= \sum_{n\ge0}e^{nvt-wt^2(k+1)n^2}-\sum_{n\ge0}e^{nvt-wt^2(k+1)n^2-vt(2n+1)-wt^2(k+1)(2n+1)}\\
&= \sum_{n\ge0}e^{nvt-wt^2(k+1)n^2}-\sum_{n\ge1}e^{-wt^2(k+1)n^2-vtn} .\end{aligned}\end{equation}

Subtracting a $1$ from (2.15) and then taking the Mellin transform, we compute that 
$$\begin{aligned} &\mathfrak{M}\bigg\{ \sum_{n\ge0}(e^{-wt^2};e^{-wt^2})_ne^{-wt^2n(n+1)/2}(-1)^nA_{n,k}(e^{-vt-wt^2(k+1)},e^{-wt^2})-1\bigg\}\\
&= \int_{0}^{\infty}t^{s-1}\left( \sum_{n\ge1}e^{nvt-wt^2(k+1)n^2}-\sum_{n\ge1}e^{-wt^2(k+1)n^2-vtn}\right)dt\\
&= \sum_{n\ge1}\int_{0}^{\infty}t^{s-1} \left(e^{nvt-wt^2(k+1)n^2} - e^{-wt^2(k+1)n^2-vtn}\right)dt\\
&=\sum_{n\ge1}\Bigg((2w(k+1)n^2)^{-s/2}\Gamma(s)e^{v^2/(8(k+1)w)}D_{-s}\left(-\frac{v}{\sqrt{2(k+1)w}}\right)\\
&- (2w(k+1)n^2)^{-s/2}\Gamma(s)e^{v^2/(8(k+1)w)}D_{-s}\left(\frac{v}{\sqrt{2(k+1)w}}\right)\Bigg)\\
&=(2w(k+1))^{-s/2}\zeta(s)\Gamma(s)e^{v^2/(8(k+1)w)}D_{-s}\left(-\frac{v}{\sqrt{2(k+1)w}}\right)\\
&- (2w(k+1))^{-s/2}\zeta(s)\Gamma(s)e^{v^2/(8(k+1)w)}D_{-s}\left(\frac{v}{\sqrt{2(k+1)w}}\right). \end{aligned}$$

Here we employed Lemma 2.4 with $v$ replaced by $nv,$ and $w$ replaced by $w(k+1)n^2,$ and the resulting formula is analytic for $\Re(s)>1.$ Now applying Mellin inversion, we compute that for $\Re(s)=c>1,$
\begin{equation}\begin{aligned} &-1+\sum_{n\ge0}(e^{-wt^2};e^{-wt^2})_n(-1)^ne^{-wt^2n(n+1)/2}A_{n,k}(e^{-vt-wt^2(k+1)},e^{-wt^2})\\
&=\frac{1}{2\pi i}\int_{(c)}\bigg((2w(k+1))^{-s/2}\zeta(s)\Gamma(s)e^{v^2/(8(k+1)w)}D_{-s}(-\frac{v}{\sqrt{2(k+1)w}})\\
&-(2w(k+1))^{-s/2}\zeta(s)\Gamma(s)e^{v^2/(8(k+1)w)}D_{-s}(\frac{v}{\sqrt{2(k+1)w}})\bigg)t^{-s}ds.\end{aligned}\end{equation}
The modulus of the integrand can be seen to be estimated as follows (see [11, pg.398] for a similar example). Making the change of variable $s\rightarrow s+\frac{1}{2},$ we obtain an integral for $\Re(s)>\frac{1}{2}.$ For $\Re(s)=\sigma>\frac{1}{2},$ we have $\zeta(s+\frac{1}{2})\ll \zeta(\sigma+\frac{1}{2}).$ Now an estimate of Paris [10, pg. 425, A(10)] for the parabolic cylinder function $D_{-s-\frac{1}{2}}(x)$ for fixed $x$ as $|s|\rightarrow\infty,$ says that
\begin{equation}D_{-s-\frac{1}{2}}(x)=\frac{\sqrt{\pi}e^{-x\sqrt{s}}}{2^{s/2+1/4}\Gamma(\frac{s}{2}+\frac{3}{4})}\left(1-\frac{x^3}{24\sqrt{s}}+\frac{x^2}{24s}(\frac{x^2}{48}-\frac{3}{2})+O(s^{-3/2})\right),\end{equation}
uniformly for $|\arg(s)|\le \pi-\delta<\pi.$ The growth of the integrand is then seen to be dominated by $\Gamma(s+\frac{1}{2})D_{-s-\frac{1}{2}}(\frac{v}{\sqrt{2(k+1)w}}),$ due to Stirling's formula and (2.17), and consequently decays exponentially. By the asymptotic estimate (2.17) in conjunction with [11, pg.39, Lemma 2.2], we see the growth of the integrand is well controlled. Hence, noting that the integrals along the horizontal segments of a rectangular contour tend to $0,$ we may apply Cauchy's residue theorem moving the line of integration to the left.

\par The integrand of (2.16) has simple poles at $s=1$ and the negative integers $s=-n$ due to $\Gamma(s).$ Using $\lim_{s\rightarrow1}(s-1)\zeta(s)=1,$ and [7, pg.1030, eq.(9.254),\#1] $D_{-1}(x)=\sqrt{\pi/2}e^{x^2/4}(1-\erf(\frac{x}{\sqrt{2}})),$
$$\begin{aligned} &\lim_{s\rightarrow1}(s-1)t^{-s}\bigg((2w(k+1))^{-s/2}\zeta(s)\Gamma(s)e^{v^2/(8(k+1)w)}D_{-s}(-\frac{v}{\sqrt{2(k+1)w}})\\
&-(2w(k+1))^{-s/2}\zeta(s)\Gamma(s)e^{v^2/(8(k+1)w)}D_{-s}(\frac{v}{\sqrt{2(k+1)w}})\bigg)\\
&=t^{-1}(2w(k+1))^{-1/2}e^{v^2/(8(k+1)w)}\bigg(D_{-1}(-\frac{v}{\sqrt{2(k+1)w}})-D_{-1}(\frac{v}{\sqrt{2(k+1)w}})\bigg)\\
&=t^{-1}\sqrt{\frac{\pi}{4w(k+1)}}e^{v^2/(4(k+1)w)}\bigg(\left(1-\erf\left(-\frac{v}{\sqrt{4(k+1)w}})\right)\right) \\
&-\left(1-\erf\left(\frac{v}{\sqrt{4(k+1)w}})\right)\right)\bigg).\end{aligned}$$

Displacing the contour to the left, we compute the residues at the negative integers, and obtain
$$\begin{aligned}&-1+\sum_{n\ge0}(e^{-wt^2};e^{-wt^2})_n(-1)^ne^{-wt^2n(n+1)/2}A_{n,k}(e^{-vt-wt^2(k+1)},e^{-wt^2})\\
&=t^{-1}\sqrt{\frac{\pi}{4w(k+1)}}e^{v^2/(4(k+1)w)}\bigg(\erf\left(\frac{v}{2\sqrt{(k+1)w}}\right)-\erf\left(-\frac{v}{2\sqrt{(k+1)w}}\right)\bigg)\\
&+\sum_{N\ge n\ge0}\frac{(2w(k+1))^{n/2}}{n!}(-t)^n\zeta(-n)e^{v^2/(8(k+1)w)}D_{n}\left(-\frac{v}{\sqrt{2(k+1)w}}\right)\\
&- \sum_{N\ge n\ge0}\frac{(2w(k+1))^{n/2}}{n!}(-t)^n\zeta(-n)e^{v^2/(8(k+1)w)}D_{n}\left(\frac{v}{\sqrt{2(k+1)w}}\right)\\
&+R_N.\end{aligned}$$
We estimate the remainder $R_N$ by considering the the underlying integral
$$\bar{R}_N:=\frac{1}{2\pi i}\int_{-2N-i\infty}^{-2N+i\infty}t^{-s}\zeta(s)\Gamma(s)D_{-s}(z)ds.$$ Putting $s=-2N+iu$ and applying [15, pg.16, eq.(2.1.8)] $\Gamma(s)\zeta(s)=\zeta(1-s)2^{s-1}\pi^{s}\sec(\frac{s\pi}{2})$ we have that for $\arg(t)<\frac{\pi}{2},$
$$\begin{aligned}|\bar{R}_N|&\le \frac{|t|^{2N}}{2(2\pi)^{2N}}\int_{-\infty}^{\infty}\frac{|\zeta(2N+1+iu)|}{|\cos\left( (iu-2N)\frac{\pi}{2}\right)|}|D_{2N-iu}(z)|e^{\arg(t)u}du\\
& \le \frac{|t|^{2N}\zeta(2N+1)}{2(2\pi)^{2N}}\int_{-\infty}^{\infty}\frac{|D_{2N-iu}(z)|}{|\cos\left( (iu-2N)\frac{\pi}{2}\right)|}e^{\arg(t)u}du.\end{aligned}$$
To estimate $D_{2N-iu}(z),$ we apply an integral obtained from [7, pg.1028, eq.(9.241)\#2] with the linear relation (corrected) [7, pg.1030, eq.(9.248)\#1] $D_s(z)=\Gamma(s+1)(e^{is\pi/2}D_{-s-1}(iz)+e^{-is\pi/2}D_{-s-1}(-iz))/\sqrt{2\pi},$ which is
$$D_s(z)=\sqrt{\frac{2}{\pi}}e^{z^2/4}\int_{0}^{\infty}\cos(zy-\frac{\pi}{2}s)y^s e^{-y^2/2}dy.$$
Putting $s=2N+iu,$ with $|\cos(zy-\frac{\pi}{2}(2N+iu))|\le \cosh(\frac{\pi}{2}u)$ if $z$ is real, we have
$$\begin{aligned}  |D_{2N-iu}(z)|&\le \sqrt{\frac{2}{\pi}}e^{z^2/4}\int_{0}^{\infty}|\cos(zy-\frac{\pi}{2}(2N+iu))||y^{2N+iu}|e^{-y^2/2}dy\\
& \le \sqrt{\frac{2}{\pi}}e^{z^2/4}\cosh(\frac{\pi}{2}u)\int_{0}^{\infty}y^{2N}e^{-y^2/2}dy\\
&= \sqrt{\frac{1}{\pi}}e^{z^2/4}2^{N}\cosh(\frac{\pi}{2}u)\Gamma(2N+\frac{1}{2}).\end{aligned}$$

Consequently, collecting these estimates we see that $R_N\rightarrow\infty$ as $N\rightarrow\infty$ which implies the expansion is asymptotic. Therefore, we have as $t\rightarrow0^{+},$
$$-1+\sum_{n\ge0}(e^{-wt^2};e^{-wt^2})_n(-1)^ne^{-wt^2n(n+1)/2}A_{n,k}(e^{-vt-wt^2(k+1)},e^{-wt^2})$$
$$\sim t^{-1}\sqrt{\frac{\pi}{4w(k+1)}}e^{v^2/(4(k+1)w)}\bigg(\erf\left(\frac{v}{2\sqrt{(k+1)w}}\right)-\erf\left(-\frac{v}{2\sqrt{(k+1)w}}\right)\bigg)$$
$$+\sum_{n\ge0}\frac{(2w(k+1))^{n/2}}{n!}(-t)^n\zeta(-n)e^{v^2/(8(k+1)w)}D_{n}\left(-\frac{v}{\sqrt{2(k+1)w}}\right)$$
$$- \sum_{n\ge0}\frac{(2w(k+1))^{n/2}}{n!}(-t)^n\zeta(-n)e^{v^2/(8(k+1)w)}D_{n}\left(\frac{v}{\sqrt{2(k+1)w}}\right).$$
The resulting single sum in the theorem arises from the parity relationship $D_n(-x)=(-1)^nD_n(x),$ and cancellation of terms.
\end{proof}

\begin{proof}[Proof of Theorem~\ref{thm:theorem2}] 
Inserting (2.8)--(2.9) (with $q$ replaced by $q^{2^k}$) into Lemma 2.2 gives us the following Bailey pair
\begin{equation} \hat{\beta}_n(q,q)=\sum_{n\ge r_1\ge r_2\dots\ge r_k\ge0}P_{n,r_1,r_2,\dots,r_k}(q)q^{n-r_1+2(r_1-r_2)+\dots+2^{k-1}(r_{k-1}-r_k)},\end{equation}
where
$$P_{n,r_1,r_2,\dots,r_k}(q):=\frac{(-q^2;q)_{2r_1}(-q^{2^2};q^{2})_{2r_2}\dots(-q^{2^{k}};q^{2^{k-1}})_{2r_k}}{(q^2;q^2)_{n-r_1}(q^{2^2};q^{2^2})_{r_1-r_2}\dots(q^{2^{k}};q^{2^{k}})_{r_{k-1}-r_{k}}}\frac{(z;q^{2^k})_{r_k+1}(q^{2^k}/z;q^{2^k})_{r_k}}{(q^{2^k+1};q^{2^k})_{2r_k}},$$

\begin{equation}\hat{\alpha}_n(q,q)=(-z)^{-n}q^{2^{k}n(n+1)/2}(1-z^{2n+1}),\end{equation}
Inserting (2.18)--(2.19) into (2.13) gives
\begin{equation}  \sum_{n\ge0}(q)_n(-1)^nq^{n(n+1)/2}B_{n,k}(z,q)=\sum_{n\ge0}(-1)^n(-z)^{-n}q^{(2^{k}+1)n(n+1)/2}(1-z^{2n+1}).\end{equation}
Putting $q=e^{-wt^2},$ and $z=e^{-vt-wt^2(2^{k}+1)/2},$ we have that (2.20) becomes
\begin{equation} \begin{aligned} &\sum_{n\ge0}(e^{-wt^2};e^{-wt^2})_n(-1)^ne^{-wt^2n(n+1)/2}B_{n,k}(e^{-vt-wt^2(2^{k}+1)/2},e^{-wt^2})\\
&=\sum_{n\ge0}e^{nvt-wt^2(2^{k}+1)n^2/2}(1-e^{-vt(2n+1)-wt^2(2^{k}+1)(2n+1)/2})\\
&= \sum_{n\ge0}e^{nvt-wt^2(2^{k}+1)n^2/2}-\sum_{n\ge0}e^{-wt^2(2^{k}+1)(n+1)^2/2-vt(n+1)}\\
&= \sum_{n\ge0}e^{nvt-wt^2(2^{k}+1)n^2/2}-\sum_{n\ge1}e^{-wt^2(2^{k}+1)n^2/2-vtn} .\end{aligned}\end{equation}

Subtracting a $1$ from (2.21) and then taking the Mellin transform, we compute that 
$$\begin{aligned} &\mathfrak{M}\bigg\{ \sum_{n\ge0}(e^{-wt^2};e^{-wt^2})_ne^{-wt^2n(n+1)/2}B_{n,k}(e^{-vt-wt^2(2^{k}+1)},e^{-wt^2})-1\bigg\}\\
&= \int_{0}^{\infty}t^{s-1}\left( \sum_{n\ge1}e^{nvt-wt^2(2^{k}+1)n^2/2}-\sum_{n\ge1}e^{-wt^2(2^{k}+1)n^2/2-vtn}\right)dt\\
&= \sum_{n\ge1}\int_{0}^{\infty}t^{s-1} \left(e^{nvt-wt^2(2^{k}+1)n^2/2} - e^{-wt^2(2^{k}+1)n^2/2-vtn}\right)dt\\
&=\sum_{n\ge1}\Bigg((w(2^{k}+1)n^2)^{-s/2}\Gamma(s)e^{v^2/(4(2^{k}+1)w)}D_{-s}\left(-\frac{v}{\sqrt{(2^{k}+1)w}}\right)\\
&- (w(2^{k}+1)n^2)^{-s/2}\Gamma(s)e^{v^2/(4(2^{k}+1)w)}D_{-s}\left(\frac{v}{\sqrt{2(2^{k}+1)w}}\right)\Bigg)\\
&=(w(2^{k}+1))^{-s/2}\zeta(s)\Gamma(s)e^{v^2/(4(2^{k}+1)w)}D_{-s}\left(-\frac{v}{\sqrt{(2^{k}+1)w}}\right)\\
&- (w(2^{k}+1))^{-s/2}\zeta(s)\Gamma(s)e^{v^2/(4(2^{k}+1)w)}D_{-s}\left(\frac{v}{\sqrt{2(2^{k}+1)w}}\right). \end{aligned}$$

Here we employed Lemma 2.4 with $v$ replaced by $nv,$ and $w$ replaced by $w(2^{k}+1)n^2/2,$ and again the resulting formula is analytic for $\Re(s)>1.$ By Mellin inversion, we compute for $\Re(s)=c>1,$
$$\begin{aligned} &-1+\sum_{n\ge0}(e^{-wt^2};e^{-wt^2})_n(-1)^ne^{-wt^2n(n+1)/2}B_{n,k}(e^{-vt-wt^2(2^{k}+1)},e^{-wt^2})\\
&=\frac{1}{2\pi i}\int_{(c)}\bigg((w(2^{k}+1))^{-s/2}\zeta(s)\Gamma(s)e^{v^2/(4(2^{k}+1)w)}D_{-s}\left(-\frac{v}{\sqrt{(2^{k}+1)w}}\right)\\
&- (w(2^{k}+1))^{-s/2}\zeta(s)\Gamma(s)e^{v^2/(4(2^{k}+1)w)}D_{-s}\left(\frac{v}{\sqrt{(2^{k}+1)w}}\right)\bigg)t^{-s}ds.\end{aligned}$$
The integrand has a simple pole at $s=1$ and the negative integers $s=-n$ due to $\Gamma(s).$ We compute the residues at the negative integers to see that as $t\rightarrow0^{+},$
$$-1+\sum_{n\ge0}(e^{-wt^2};e^{-wt^2})_n(-1)^ne^{-wt^2n(n+1)/2}B_{n,k}(e^{-vt-wt^2(2^{k}+1)},e^{-wt^2})$$
$$\sim t^{-1}\sqrt{\frac{\pi}{4w(2^k+1)}}e^{v^2/(4(2^k+1)w)}\bigg(\erf\left(\frac{v}{2\sqrt{(2^k+1)w}}\right)-\erf\left(-\frac{v}{2\sqrt{(2^k+1)w}}\right)\bigg)$$
$$+\sum_{n\ge0}\frac{(w(2^{k}+1))^{n/2}}{n!}(-t)^n\zeta(-n)e^{v^2/(4(2^{k}+1)w)}D_{n}\left(-\frac{v}{\sqrt{(2^{k}+1)w}}\right)$$
$$- \sum_{n\ge0}\frac{(w(2^{k}+1))^{n/2}}{n!}(-t)^n\zeta(-n)e^{v^2/(4(2^{k}+1)w)}D_{n}\left(\frac{v}{\sqrt{(2^{k}+1)w}}\right).$$
The resulting single sum in the theorem again arises from the parity relationship $D_n(-x)=(-1)^nD_n(x)$ and cancellation of terms.
\end{proof}

\begin{proof}[Proof of Theorem~\ref{thm:theorem3}] 
Since the proof is identical to the proof of our previous two theorems, we only outline some of the major details. Inserting the Bailey pair (2.8)--(2.9) into Lemma 2.3 and using (2.13) we obtain the identity
\begin{equation} \sum_{n\ge0}\frac{(q)_n}{(-q)_{n}}(-1)^nq^{n(n+1)/2}C_{n,k}(z,q)\end{equation}
$$=\sum_{n\ge0}z^{-n}q^{(k+2)n(n+1)/2}(1-z^{2n+1}),$$
where
$$C_{n.k}(z,q)=\sum_{n\ge r_1\ge r_2\dots\ge r_k\ge0}\frac{q^{r_1^2/2+r_1/2+r_2^2/2+r_2/2+\dots+r_k^2/2+r_k/2}}{(q)_{n-r_1}(q)_{r_1-r_2}\dots(q)_{r_{k-1}-r_{k}}}\frac{(z)_{r_k+1}(q/z)_{r_k}}{(q)_{{r_k}}(q;q^2)_{r_k+1}}.$$
Putting $q=e^{-wt^2},$ and $z=e^{-vt-wt^2(k+2)/2},$ and proceeding as before the result follows. \end{proof}

\begin{proof}[Proof of Theorem~\ref{thm:theorem4}] 
By Lemma 2.4, we have for $\Re(s)=c'>1,$ $\Re(w)>0,$
$$\sum_{n\ge1}\chi(n)e^{-w nt^2-v nt}=\frac{1}{2\pi i}\int_{(c')}(2w)^{-s/2}\Gamma(s)L(s,\chi)e^{v^2/(8w)}D_{-s}(\frac{v}{\sqrt{2w}})t^{-s}ds.$$
Since $L(s,\chi)$ is assumed to be nonprincipal, there is no pole at $s=1.$ Hence, computing the residues at the poles $s=-n$ from the gamma function,
$$\sum_{n\ge1}\chi(n)e^{-w n^2t^2-v nt}\sim e^{v^2/(8w)}\sum_{n\ge0}\frac{(2w)^{n/2}(-t)^n}{n!}L(-n,\chi)D_{n}(\frac{v}{\sqrt{2w}}),$$
as $t\rightarrow0^{+}.$
After noting that if $\chi$ is even, then $L(-2n,\chi)=0,$ and if $\chi$ is odd, then $L(-2n-1,\chi)=0$ the result follows. 

\end{proof}
\section{Concluding Remarks}

Here we have of course limited ourselves with our choices of changing base of $q$ to three examples from [5]. Therefore, many more examples may be obtained by appealing to different Bailey chains from [5]. It would be desirable to obtain expansions for other $L$-functions, such as those contained in [3].

1390 Bumps River Rd. \\*
Centerville, MA
02632 \\*
USA \\*
\\*
ul. A. E. Ody\'{n}ca 47 \\*
02-606 Warsaw\\*
Poland\\*
E-mail: alexpatk@hotmail.com, alexepatkowski@gmail.com

\end{document}